\documentclass{amsart}

\usepackage{graphicx, amssymb, amsmath,amsthm,mathtools,enumitem,yhmath,todonotes,verbatim, url}
\usepackage[T1]{fontenc}
\usepackage[osf]{mathpazo}
\usepackage{eucal}
\usepackage{epigraph}
\usepackage[colorlinks=true,linkcolor=cyan,citecolor=cyan]{hyperref}

\newtheorem*{main_theorem:1}{Main Theorem 1}
\newtheorem*{main_theorem:2}{Main Theorem 2}
\newtheorem{theorem}{Theorem}[section]
\newtheorem{lemma}[theorem]{Lemma}
\newtheorem{corollary}[theorem]{Corollary}
\newtheorem{proposition}[theorem]{Proposition}
\newtheorem{question}[theorem]{Question}

\newtheorem{fact}[theorem]{Fact}

\theoremstyle{definition}
\newtheorem{definition}[theorem]{Definition}

\newtheorem{example}[theorem]{Example}

\theoremstyle{remark}

\numberwithin{equation}{section}

\newenvironment{enumerate-(a)}{\begin{enumerate}[label={\upshape (\alph*)}, leftmargin=2pc]}{\end{enumerate}}
\newenvironment{enumerate-(a)-r}{\begin{enumerate}[label={\upshape (\alph*)}, leftmargin=2pc,resume]}{\end{enumerate}}
\newenvironment{enumerate-(a)-5}{\begin{enumerate}[label={\upshape (\alph*)}, leftmargin=2pc,start=5]}{\end{enumerate}}
\newenvironment{enumerate-(A)}{\begin{enumerate}[label={\upshape (\Alph*)}, leftmargin=2pc]}{\end{enumerate}}
\newenvironment{enumerate-(A)-r}{\begin{enumerate}[label={\upshape (\Alph*)}, leftmargin=2pc,resume]}{\end{enumerate}}
\newenvironment{enumerate-(i)}{\begin{enumerate}[label={\upshape (\roman*)}, leftmargin=2pc]}{\end{enumerate}}
\newenvironment{enumerate-(i)-r}{\begin{enumerate}[label={\upshape (\roman*)}, leftmargin=2pc,resume]}{\end{enumerate}}
\newenvironment{enumerate-(I)}{\begin{enumerate}[label={\upshape (\Roman*)}, leftmargin=2pc]}{\end{enumerate}}
\newenvironment{enumerate-(I)-r}{\begin{enumerate}[label={\upshape (\Roman*)}, leftmargin=2pc,resume]}{\end{enumerate}}
\newenvironment{enumerate-(1)}{\begin{enumerate}[label={\upshape (\arabic*)}, leftmargin=2pc]}{\end{enumerate}}
\newenvironment{enumerate-(1)-r}{\begin{enumerate}[label={\upshape (\arabic*)}, leftmargin=2pc,resume]}{\end{enumerate}}

\DeclareMathOperator{\supp}{supp}

\title{Compactness for almost projective modules}

\author[F. Calderoni]{Filippo Calderoni}
\address{Department of Mathematics, Rutgers University, Hill Center for the Mathematical
Sciences, 110 Frelinghuysen Rd., Piscataway, NJ 08854-8019}
\email{filippo.calderoni@rutgers.edu}

\author[E.~Yanowitz]{Esm\'e Yanowitz}
\address{Department of Mathematics, Brown University, Providence, RI 02912 USA.}
\email{esme\_yanowitz@brown.edu}

\begin{document}

\begin{abstract}
In this paper we contribute to the study of compactness properties for algebraic structures. We prove two compactness theorems in the context of almost projective modules from \(\lambda\)-strongly compact and weakly compact, respectively.
\end{abstract}

\maketitle

\section{Introduction}

The application of set-theoretic methods, and particularly large cardinals, has been fruitful in both infinite abelian group theory and in the theory of almost free modules. Some of the most striking results in this area revolve around the notion of compactness. (E.g., see~\cite{Magidor_Shelah_1994}.)

Roughly speaking, compactness is a property in which a structure is characterized by
its local behavior. Such a phenomenon is the result of some set theoretical assumptions of largeness.  For example, a celebrated compactness theorem of Shelah~\cite{Shelah_1975} shows that singularity implies compactness for almost free abelian groups. More precisely, let \(\kappa\) be a singular cardinal and \(G\) be an abelian group with \(|G|=\kappa\). If all subgroups of \(G\) with cardinality less than \(\kappa\) are free, then \(G\) is free. One obtains the same conclusion if \(\kappa\) satisfies large cardinal assumptions. The case when \(\kappa\) is weakly compact is thoroughly discussed in the monograph of Eklof and Mekler~\cite[Section~IV]{Eklof_Mekler_2002} and in the more recent book of Fuchs~\cite{Fuchs_15}. For an account of compactness results we refer the reader to the survey article of Honzik~\cite{Honzik_2026}. Isolating new compactness theorems is the main motivation for this work.

While compactness results of free abelian groups have been generalized to free \(R\)-modules in \cite{Eklof_Mekler_2002}, it is more challenging to consider less restrictive assumptions, such as looking at projective modules . In this work, we concentrate on projectivity and prove the following main results:

\begin{theorem}
\label{result!}
    Let $\kappa$ be a weakly compact cardinal. If $M$ is a $\leq\! \kappa$-generated module which is $\kappa$-projective, then $M$ is projective.
\end{theorem}

\begin{theorem}
\label{strongresult!}
     Let \(\lambda < \kappa\) with \(\kappa\) \(\lambda\)-strongly compact. Let \(R\) be a left perfect, right coherent ring of cardinality \(< \! \lambda\). If \(M\) is a \(\kappa\)-projective  \(R\)-module,  \(M\) is projective.
\end{theorem}


Previously, Calderoni and Ostrem~\cite{Calderoni_Ostrem} obtained analogous compactness results for
\(\Sigma\)-cyclic groups (direct sums of cyclic groups). In our discussion we develop similar techniques for almost projective \(R\)-modules. (As every abelian group can be regarded as a \(\mathbb{Z}\)-module our analysis has broader scope.)

In Section~\ref{Section: Preliminaries} we cover background and preliminary results.
In Section~\ref{Section: Weak Compactness} we prove Theorem~\ref{result!}. The proof uses \(\Gamma\)-invariants. These methods were first developed by Eklof~\cite{Eklof_1993} and employed in various cases by ~\cite{Baumgartner_Foreman_Spinas_1997,Eklof_Trlifaj_1998,Calderoni_Ostrem, Mihara_2026}.
In Section~\ref{Section: Strong Compactness} we prove Theorem~\ref{strongresult!} using ultraproducts; our discussion involves classical results regarding elementarity of modules from Eklof and Sabbagh~\cite{Sabbagh_Eklof_1971}.




\subsection*{Acknowledgments}
The first author is supported by NSF grant DMS-2348819.
The second author participated in the 2026 REU Program
coordinated by the Center for Discrete Mathematics and Theoretical Computer
Science (DIMACS). We thank DIMACS for providing resources and a
stimulating environment to conduct this research. We thank Cecelia Higgins for useful comments and helpful remarks.

\section{Preliminaries}
\label{Section: Preliminaries}

Throughout the paper \(R\) is a ring with identity and not necessarily commutative.  Unless otherwise specified, \(R\)-modules are unitary left modules. If \(M\) is an \(R\)-module and \(Y \subset M \), then we denote by \(\langle Y \rangle\) the submodule of \(M\) generated the subset \(Y\). If \(|Y| \leq \kappa\), we say that \(\langle Y \rangle\) is \emph{\( \leq \! \kappa\)-generated}.

Now we turn to discuss free and projective modules and their properties. 

\begin{definition}\label{definition:Free_Module}
    An \(R\)-module \(M\) is called \emph{free} if there exists a generating set \(B\) for \(M\) such that, for any \(R\)-module \(N\), every set map \(f \colon B \to N\)  extends to a homomorphism of modules \(\overline{f} \colon M \to N\). In this case, the homomorphism \(\overline{f}\) is necessarily unique and \(B\) is called a \emph{basis} of \(M\). 
\end{definition}

    Recall that a \emph{short exact sequence} of modules consists of an injective homomorphism of modules \(\phi \colon M \to M'\) and a surjective homomorphism of modules \(\psi \colon M' \to M''\) such that \(\operatorname{im} \phi = \ker \psi\). Such a sequence is denoted by 
        \[0 \to M \stackrel{\phi}{\to} M' \stackrel{\psi}{\to} M'' \to 0\]
  and is said to \emph{split} if there exists a map \(f \colon M'' \to M'\) such that \(\psi \circ f = \operatorname{id}_{M''}\). In this case, \(M' \cong M \oplus M''\).

\begin{definition}\label{definition:Projective_Module}
   An \(R\)-module \(P\) is called \emph{projective} if every short exact sequence of the form
        \[0 \to M \to N \to P \to 0\]
  splits.
\end{definition}

Equivalently, an \(R\)-module \(P\) is projective if it satisfies the \emph{lifting property}. That is,
for every surjective \(R\)-module homomorphism \(f \colon N \to M\) and every module homomorphism \(g \colon P \to M\), there exists a module homomorphism \(h \colon P \to N\) such that \(f\circ h = g\).

We also note that for any \(R\)-module \(M\) (and every ring \(R\)), there exists a free \(R\)-module \(F\) and a surjective module homomorphism \(f \colon F \to M\). Therefore, any projective \(R\)-module is a free summand of a free \(R\)-module.


It is not hard to see that every free \(R\)-module satisfies the lifting property. Therefore, for any ring \(R\), free \(R\) modules are projective.

Now let \(R\) be a principal ideal domain (p.i.d.). Recall that every submodule  of a free \(R\)-module is again free. (See~\cite[Chapter~III, Theorem~7.1]{Lang}.) As any projective \(R\)-module \(M\) is submodule of a free \(R\)module, an \(R\)-module \(M\) is projective if and only if \(M\) is free.
In particular, this is the case for \(R=\mathbb{Z}\), i.e., for abelian groups.

Projective modules enjoy a number of closure properties. For instance, direct sums and summands of projective modules are again projective.

In Section~\ref{Section: Weak Compactness} we use a corollary of the so-called \emph{Kaplansky's d\'evissage}, a famous decomposition theorem of Kaplansky~\cite{Kaplansky_1958}. 

\begin{theorem}[Kaplansky, 1958]\label{theorem:Kaplansky}
    Let \(M\) be an \(R\)-module which is a direct sum of countably generated \(R\)-modules. Then any direct summand of \(M\) is a direct sum of countably generated \(R\)-modules. 
\end{theorem}

We recall that every projective module is the direct summand of some free module. Since each free $R$-module is the direct sum of some number of copies of $R$ an application of Theorem~\ref{theorem:Kaplansky} yields the following corollary.

\begin{corollary}\label{corollary:Kaplansky}
    Every projective module is a direct sum of countably generated projective modules. 
\end{corollary}

In this paper we will deal with the notion $\kappa$-projective, which is an analogue of $\kappa$-free (e.g, see \cite{Eklof_Mekler_2002}). 

\begin{definition}\label{definition:Almost-Projective_or_Kappa-Projective}
    A \(\leq \!\kappa\)-generated \(R\)-module \(M\) is \emph{$\kappa$-projective} if and only if there is a subset \(\mathcal{C}\) of the power set on \(M\) satisfying:
    \begin{enumerate-(i)}
        \item Every element of \(\mathcal{C}\) is a \(<\!\kappa\)-generated projective submodule of \(M\); 
        \item Every subset of \(M\) of cardinality \(<\!\kappa\) is contained in an element of \(\mathcal{C}\);
        \item The set \(\mathcal{C}\) is closed under unions of well-ordered chains of length \(< \!\kappa\).
    \end{enumerate-(i)}
    The terms \emph{almost projective} and \emph{\(\kappa\)-projective} may be used interchangeably.
\end{definition}

\begin{definition}\label{definition:Kappa-Filtration}
    Let \(M\) be a \(\leq \!\kappa\)-generated \(R\)-module. A \emph{\(\kappa\)-filtration} of \(M\) is a indexed set \(\{M_{\alpha} : \alpha < \kappa\}\) of submodules of \(M\) such that 
        \begin{enumerate-(i)}
            \item \(M = \bigcup_{\alpha < \kappa}M_{\alpha}\);
            \item Each \(M_{\alpha}\) is a \(<\! \kappa\)-generated submodule of \(M\);
            \item If \(\beta < \alpha\) then \(M_{\beta} \subseteq M_{\alpha}\);
            \item If \(\alpha\) is a limit ordinal then \(M_{\alpha} = \bigcup_{\beta < \alpha}M_{\beta}\).
        \end{enumerate-(i)}
\end{definition}

Whenever \(M\) is a \(\leq\!\kappa\)-generated \(\kappa\)-projective \(R\) module it is easy to construct a \(\kappa\)-filtration \(\{M_\alpha:\alpha<\kappa\}\) consisting of projective \(R\)-modules. In fact, we have the following characterization:

\begin{proposition}\label{lemma:Almost-Projective_or_Kappa-projective}
A \(\leq\!\kappa\)-generated \(R\)-module \(M\) is \(\kappa\)-projective if and only if  \(M\) has a \(\kappa\)-filtration consisting of projective modules.
\end{proposition}

\section{Compactness}
\label{Section: Weak Compactness}


In this section we prove Theorem~\ref{result!}. Our argument is purely set-theoretical. It  builds on the work of Eklof~\cite{Eklof_1993} and mimics the proof of \cite[Theorem~1.2]{Calderoni_Ostrem}.

First we recall some definitions in infinite combinatorics. Let \(\gamma\) be an infinite limit ordinal.

\begin{definition}
    A subset \(C\) of \(\gamma\) is said to be a \emph{club}  if and only if:
    \begin{enumerate-(i)}
    \item \(C\) is \emph{closed} in \(\gamma\), i.e., for all \(A \subseteq C\), if \( \sup A <\gamma\), then
\(\sup A \in C\); and
 \item \(C\) is \emph{unbounded} in \(\gamma\), i.e., for every \(\alpha<\gamma\) there is some \(\beta \in C\) with \(\alpha<\beta\).
\end{enumerate-(i)}

Moreover, a subset \(S\subseteq \gamma\) is said to be \emph{stationary} if and only if it intersects every club \(C\subseteq S\). I.e., for all clubs \(C\subseteq \gamma\), we have \(S \cap C \neq \emptyset\).
\end{definition}

Given a family of subsets \(\{A_\alpha: \alpha < \kappa\}\)
of \(\kappa\), we define the \emph{diagonal intersection} of the family by

\[
\bigtriangleup_{\alpha<\kappa} A_\alpha \coloneqq\left\{\beta\in \kappa : \beta \in \bigcup_{\alpha<\beta} A_\alpha \right\}.
\]

The following is a well-known closure property for clubs. To see the proof we refer the reader to \cite[Lemma~8.4]{Jech_2002}.

\begin{fact}
\label{fact : diagonal intersection}
    If \(\{C_\alpha : \alpha < \kappa\}\) is a family of clubs of \(\kappa\), then the diagonal intersection \(\bigtriangleup_{\alpha<\kappa} C_\alpha\)  is also a club in \(\kappa\).
\end{fact}

An uncountable cardinal \(\kappa\) is said to be \emph{weakly compact} if for every function
\(f\colon [\kappa]^2 \to \{0, 1\}\), there exists a subset \(A \subseteq \kappa\)
of cardinality \(|A| = \kappa\) such that \(f\) is constant on \([A]^2\) (meaning all unordered pairs in \(A\) are mapped by \(f\) to \(0\), or all are mapped to \(1\)).

Weakly compact cardinals are well-known for their interesting combinatorial properties. The one stated below is a form of stationary reflection, and it will be used later in this section. For a proof we refer the reader to Kanamori’s monograph
\cite[Proposition 4.1]{Kanamori_2003}.

\begin{lemma}
\label{theorem:stationary_reflection}
        Let \(\kappa\) be a weakly compact cardinal and let \(\{S_\alpha: \alpha<\kappa\}\) be a family of stationary subsets of \(\kappa\). Then there is a stationary \(T\subseteq \kappa\) consisting of regular cardinals such that for all \(\lambda \in T\) and \(\alpha<\lambda\), the set
        \( S_\alpha \cap \lambda\) is stationary in \(\lambda\).
    \end{lemma}

Let \(A,B\subseteq \kappa\), we say that \(A\) and \(B\) are equivalent (in symbols, \(A\sim B\)) if and only if
there is a club \(C\subseteq \kappa\) such that \(A\cap C = B\cap C\).

\begin{definition}\label{definition:Gamma_Invariant}
Let $R$ be a ring of cardinality less than $\kappa$, and  $M$ be a \(\leq \kappa\)-generated $R$-module. Fix  a $\kappa$-filtration  $\{M_{\alpha} : \alpha < \kappa\}$ of $M$. Define the \emph{$\Gamma$-invariant} of $M$, which we denote by $\Gamma(M)$, to be the equivalence class of the set
    \begin{equation}
    \label{equation : E}
        E \coloneqq \big\{\alpha \in \kappa : \{\beta > \alpha : M_{\beta}/M_{\alpha} \text{ is not projective } \} \text { is stationary in } \kappa \big\}
    \end{equation}
modulo \(\sim\). With a slight abuse of language, we will write \(\Gamma(M) = E\) instead of \(\Gamma(M) =[E]_\sim\). In particular, we will write  \(\Gamma(M) = 0\) (or \(\Gamma(M) = \emptyset\)), when \(E\) is not stationary.
\end{definition}

We emphasize the following general properties about filtrations.

\begin{lemma}
\label{lemma : filtrations are equivalent}
    Let \(\kappa\) be a regular uncountable cardinal and let \(M\) be a \(\leq\! \kappa\)-generated module. Then 
    \begin{enumerate-(i)}
        \item M has a $\kappa$-filtration;
        \item
        \label{item : filtration club}
        If $\{M_{\alpha} : \alpha < \kappa\}$ and $\{M'_{\alpha} : \alpha < \kappa\}$ are two $\kappa$-filtrations of $M$ there is a club $C$ in $\kappa$ such that, for $\alpha \in C$, $M_{\alpha} = M'_{\alpha}$.
    \end{enumerate-(i)}
\end{lemma}

Note that Lemma~\ref{lemma : filtrations are equivalent}\ref{item : filtration club} implies that the definition of \(\Gamma(M)\) does not depend on the filtration.

The following theorem and corollary, from \cite{Eklof_1993}, demonstrate that for an arbitrary $R$-module $M$ the condition $\Gamma(M) = 0$  implies that $M$ is ``very close" to being projective. In the case that $M$ is $\kappa$-projective, $\Gamma(M) = 0$ is equivalent to $M$ being projective. We include detailed proofs for completeness.

\begin{theorem}[Eklof, 1993]\label{theorem:Gamma(M)=0if_and_only_if_M=N_oplus_P_where_N_< kappa_and_P_is_projective}
    \label{lemma : Gamma(M)=0}
    $\Gamma(M) = 0$ if and only if  $M = N \oplus P$ for some modules $N$, $P$, where $|N| < \kappa$ and $P$ is projective.
\end{theorem}

    \begin{proof}
        Suppose that $M = N \oplus P$ where $|N|< \kappa $ and $P$ is projective. By Theorem~\ref{corollary:Kaplansky} we may write $P = \bigoplus_{i \in I} P_i$ as a direct sum of countably-generated projective submodules $P_i$. Then we may form a $\kappa$-filtration on $M$ consisting of modules isomorphic to $N  \oplus (\oplus_{i \in I_{\alpha}}P_i)$ for some subset $I_\alpha \subseteq I$. Fix $\gamma < \kappa $. Note that for $\gamma < \beta < \kappa $ the quotient
        \[\frac{N \oplus (\bigoplus_{i \in I_\beta} P_i)}{N \oplus (\bigoplus_{i \in I_\gamma}P_i)} \cong \bigoplus_{i \in I_{\beta} \setminus I_{\gamma}}P_i\]
        is a direct sum of projective modules, therefore it is projective. Hence, the set $E$ associated with the $\kappa$-filtration on $M$ is empty. Therefore $\Gamma(M) = 0$.
        
        Conversely, suppose that $\Gamma(M) = 0$. Define $E$ as in \eqref{equation : E}. Since $E$ is not stationary, there is some club $C$ in $\kappa$ such that $E \cap C = \emptyset$. Then for all $\alpha \in C$ there exists a club $C_{\alpha}$ in $\kappa$ such that $M_{\beta}/M_{\alpha}$ is projective for each $\beta \in C_{\alpha}$. Define 
        $C_1 = C \cap \triangle_{\alpha \in C} C_{\alpha}$. Since diagonal intersections of clubs are clubs, as mentioned in Fact~\ref{fact : diagonal intersection}, it follows that $C_1$ is a club. Note that $M = \bigcup_{\alpha \in C_1}M_{\alpha}$. For each $\alpha \in C_1$  denote the successor of $\alpha$ in $C_1$ by $\alpha^+$. The quotient $M_{\alpha^+}/M_{\alpha}$ is projective. Moreover, for $\tau$ the least element of $C_1$, $M = M_\tau \oplus (\oplus_{\alpha \in C_1} M_{\alpha^+}/M_{\alpha})$. Setting $M_{\tau} = N$ and $\oplus_{\alpha \in C_1} M_{\alpha^+}/M_{\alpha} = P$ we obtain the statement.
    \end{proof}
    
\begin{corollary}\label{corollary:If_M_is_almost_projective_then_Gamma(M)_=_0_if_and_only_if_M_is_projective}
    If $M$ is $\kappa$-projective, then $\Gamma(M) = 0$ if and only if $M$ is projective.
\end{corollary}

    \begin{proof}
        If $M$ is projective, then $\Gamma(M) = 0$ by Theorem~\ref{lemma : Gamma(M)=0}. Conversely, suppose that $\Gamma(M) = 0$. Then by Theorem~\ref{lemma : Gamma(M)=0} we have that $M = N \oplus P$ for some projective module $P$ and module $N$ such that $|N|< \kappa$. Let $\{M_{\alpha}: \alpha < \kappa\}$ be a $\kappa$-filtration of $M$ consisting of projective modules. Since $P$ is projective, we may write $P = \bigoplus_{i \in I}P_i$ as a direct sum of countably-generated projective submodules of $P$. Then, $M$ admits a $\kappa$-filtration of the form  $\{N \oplus (\oplus_{i \in I_{\alpha}} P_i) : \alpha < \kappa\}$. Then there is some club $C$ in $\kappa$ such that for $\alpha \in C$, we have $N \oplus (\oplus_{i \in I_{\alpha}} P_i) = M_{\alpha}$, where $M_\alpha$ is projective. As $N$ is the summand of a projective module, $N$ is projective.
    \end{proof}

We now show that for a weakly compact cardinal $\kappa$, the local property of $\kappa$-projectivity is a strong enough property to imply projectivity.

\begin{proof}[Proof of Theorem~\ref{result!}]
    Suppose for contradiction  
    that a module $M$ is a counterexample. Let $\{M_{\alpha} : \alpha < \kappa\}$ be a $\kappa$-filtration as in the proof of Lemma~\ref{lemma:Almost-Projective_or_Kappa-projective}. Without loss of generality assume that $M_0 = \{0\}$, and that $M_{\alpha}$ is $<\! (|\alpha| + \aleph_0)$-generated  for each $\alpha$ (see \cite[IV.3.2]{Eklof_Mekler_2002}). Define the set $E$ associated to this filtration as in equation~\eqref{equation : E}. Then, by Corollary~\ref{corollary:If_M_is_almost_projective_then_Gamma(M)_=_0_if_and_only_if_M_is_projective} we have $\Gamma(M) \neq 0$. This means that \(E\) meets all club sets in \(\kappa\). In other words, $E$ is stationary in $\kappa$. For each $\alpha \in E$, let 
    \[S_{\alpha} = \{\beta > \alpha : \beta < \kappa \text{ and } M_{\beta}/M_{\alpha} \text{ is not projective}\}.\]
    Note that $S_\alpha$ is a stationary subset of \(\kappa\) by definition of $E$. Next, for any $\alpha \notin E$, we set $S_\alpha = E$. By Theorem~\ref{theorem:stationary_reflection} there is a stationary set $T\subseteq \kappa$ consisting of regular cardinals such that for each $\lambda \in T$ and every $\alpha < \lambda$ the intersection \[ S_{\alpha} \cap \lambda = \{\beta > \alpha : \beta < \lambda \text{ and } M_{\beta}/M_{\alpha} \text{ is not projective }\}
    \] is stationary in $\lambda$. 
    
    Now pick any \(\lambda\in T \) such that \(M_\lambda\) is projective. This can be done because \(T\) is stationary and the set \(\{\alpha< \kappa: M_\alpha\text{ is projective}\}\) is a club, since \(M\) is assumed to be \(\kappa\)-projective. Notice that for each  $\lambda \in T$, the indexed family $\{M_{\alpha} : \alpha < \lambda\}$ is a $\lambda$-filtration of $M_{\lambda}$. Observe that 
        \[\{\beta > 0 : \beta < \kappa \text{ and } M_{\beta}/M_{0} \text{ is not projective } \} \]
    is not stationary. Therefore $0 \notin E$, consequently \(S_0\) was defined to be \(E\) and $S_0 \cap \lambda = E \cap \lambda$ is stationary. Moreover, for each $\alpha \in E \cap \lambda$, the intersection $S_{\alpha} \cap \lambda$ is stationary in $\lambda$. Hence, 
        \[E \cap \lambda \subseteq \{\alpha < \lambda : \{\beta > \alpha : M_{\beta}/M_{\alpha} \text{ is not projective }\} \text{ is stationary in } \lambda \}.\]
    Therefore $\{\alpha < \lambda : \{\beta > \alpha \text{ and } M_{\beta}/M_{\alpha} \text{ is not projective }\} \text{ is stationary in } \lambda \}$ is stationary in $\lambda$, i.e., $\Gamma(M_{\lambda}) \neq 0$. Then $M_{\lambda}$ is not projective, a contradiction.
\end{proof}


\section{Strong Compactness}
\label{Section: Strong Compactness}

In this section we discuss  Theorem~\ref{strongresult!}. Our proof adapts an argument from the literature (e.g., see~\cite[Theorem~1.1]{Calderoni_Ostrem}) to the context of projective modules. The main challenge in generalizing existing results to the class of all modules is the fact that submodules of projective modules need not be projective. In addition, fixing the ring \(R\) does not guarantee that the class of \(R\)-modules is closed under ultraproducts.

We begin by recalling some definitions that are crucial to the following discussion. 
Let \(\{ M_i : i \in I\}\) be an indexed family of \(R\)-modules.
If \(x=(x_i)_{i\in I}\) is an element of the direct product \(\prod_{i\in I}M_i\), we define the \emph{support} of \(x\) by \(\supp(x)=\{i\in I: x_i\neq0\}\). For any filter \(D\) on \(I\), we define the submodule \(K_D\subseteq \prod_{i\in I} M_i\) by
\[x\in K_D \iff \{i\in I : x_i=0\}\in D.\]
The \emph{reduced product of \(\{M_i:i\in I\}\) with respect to \(D\)} is precisely the quotient \(\prod_{i\in I}M_i/K_D\). Whenever \(\mathcal{U}\) is an ultrafilter on \(I\), we say that \(\prod_{i\in I }M_i/K_\mathcal{U}\) is the \emph{ultraproduct of \(\{M_i:i\in I\}\) with respect to \(\mathcal{U}\)}. To simplify our notation, we denote ultraproducts by \(\prod_{i\in I }M_i/\mathcal{U}\).

\begin{definition}
    Let \(\lambda\leq \kappa\) be regular cardinals. We say \(\kappa\) is \emph{\(\lambda\)-strongly compact} if for any set \(I\), every \(\kappa\)-complete filter on \(I\) is contained in a \(\lambda\)-complete ultrafilter on \(I\).
\end{definition}

\begin{definition}\label{definition:Hereditary}
    A ring \(R\) is called \emph{left hereditary} if submodules of projective left \(R\)-modules are projective.\footnote{The reader might be familiar with an alternative definition of left hereditary. One can define \(R\) to be left hereditary if and only if every left ideal of \(R\) is a projective \(R\)-module. Due to a theorem of Kaplansky, it turns out that the two definitions are equivalent. (E.g., see the discussion in \cite[Section~2E]{Lam_1998}).}
    
    \end{definition}




The next lemma is the main ingredient towards proving Theorem~\ref{strongresult!}.

\begin{lemma}\label{lem:strong_compactness_result}
    Let \(\lambda < \kappa\) and \(\kappa\) be a \(\lambda\)-strongly compact cardinal. Let \(R\) be a ring of cardinality \(< \! \lambda\). Assume the following conditions 
        \begin{enumerate-(i)}
            \item
            \label{condition:(i)}
            \(R\) is hereditary;
            \item
            \label{condition:(ii)} 
            The class of projective \(R\)-modules is closed under ultraproducts.
        \end{enumerate-(i)}
    If \(M\) is a \(\kappa\)-projective  \(R\) module, then  \(M\) is projective.
\end{lemma}

    \begin{proof}
       Denote the collection of subsets of \(M\) with cardinality \(<\!\kappa\) by \(\mathcal{P}_{\kappa}(M)\) and for each \(Y \in \mathcal{P}_{\kappa}(M)\) define \(S_Y := \{X \in \mathcal{P}_{\kappa}(M)\} : Y \subseteq X \}\). Note that 
            \[F_{\kappa}(M) := \{Z \subseteq \mathcal{P}_{\kappa}(M) : S_{Y} \subseteq Z \text{ for some } Y \in \mathcal{P}_{\kappa}(M)\}\]
        is a \(\kappa \!\)-complete filter on \(M\) which, since \(\kappa\) is \(\lambda \!\)-strongly compact, extends to a \(\lambda \!\)-complete ultrafilter \(\mathcal{U}\) on \(M\). For each $Y \in \mathcal{P}_{\kappa}(M)$ the submodule $\langle Y \rangle $ of $M$ must be of cardinality $< \! \kappa$. Since $M$ is $\kappa$-projective each submodule $\langle Y \rangle $  of M is projective by imposed property (ii). The ultraproduct $(\prod_{Y \in \mathcal{P}_{\kappa}(M)} \langle Y \rangle)/ \mathcal{U}$ is then projective by the imposed property (i). 

        We define for each $m \in M$ a function $\mathbf{m} \in \prod_{Y \in \mathcal{P}_{\kappa}(M)} \langle Y \rangle $ as follows: for any $X \in \mathcal{P}_{\kappa}(M)$ set $\mathbf{m}(X) = m$ if $m \in X$ and $\mathbf{m}(X) = 0$ otherwise. Now, define a map $\phi \colon M \to (\prod_{Y \in \mathcal{P}_{\kappa}(M)}\langle Y \rangle)/\mathcal{U}$ given by $\phi(m) = \overline{\mathbf{m}}$, where $\overline{\mathbf{m}}$ is the equivalence class of the element $\mathbf{m}$ of $\prod_{Y \in \mathcal{P}_{\kappa}(M)}\langle Y \rangle$ in $(\prod_{Y \in \mathcal{P}_{\kappa}(M)}\langle Y \rangle)/\mathcal{U}$.
        Let $m, n \in M$. For $Y \in S_{\{m, n, m+n\}} \in \mathcal{U}$,
           \[\phi(m)(Y) + \phi(n)(Y) = \overline{\mathbf{m}}(Y) + \overline{\mathbf{n}}(Y) = \overline{\mathbf{(m + n)}}(Y) =\phi(m + n)(Y) \]
        and for any $r \in R$, $Y \in \mathcal{S}_{rm}$,
            \[r\phi(m)(Y) = r  \overline{\mathbf{m}}(Y) = \phi(rm).\]
        So $\phi$ is an $R$-module homomorphism.

        To show that $\ker \phi = \{0\}$ suppose $m \in \ker \phi$ is nonzero. Then, for any $y \in S_{\{m\}}$ we have 
            \[\mathbf{m}(Y) = m \neq 0 = \mathbf{0}(Y).\]
    Since $S_{\{m\}}$ is contained in $\mathcal{U}$ we have $\overline{\mathbf{m}} \neq \overline{\mathbf{0}}$. Hence $\phi$ is an embedding. Therefore $M$ is isomorphic to a submodule of the ultraproduct $(\prod_{Y \in \mathcal{P}_{\kappa}(M)} \langle Y \rangle)/ \mathcal{U}$, which is projective by condition~\ref{condition:(ii)}. So, \(M\) is projective by condition~\ref{condition:(i)}.
    \end{proof}

In the remainder of this section, we discuss for which rings \(R\) the class of \(R\)-modules satisfies condition~\ref{condition:(ii)} of Lemma~\ref{lem:strong_compactness_result}. First, we say that an $R$-module $M$ is said to be \emph{finitely presented} if there exists an exact sequence
    \[0 \to K \to F \to M \to 0\]
where $F$ is free of finite rank (i.e., the cardinality of a basis for $F$ is finite) and $K$ is finitely generated. 

\begin{definition}
    A ring $R$ is called \emph{left (right) coherent} if every finitely generated left (right) ideal of $R$ is finitely presented. 
\end{definition}

If \(M\) is an \(R\)-module we say a submodule \(S \subseteq M\) is \emph{superfluous} if, for any submodule \(N \subseteq M\) such that \(S \oplus N = M\) then \(N = M\). A \emph{projective cover} for \(M\) is an epimorphism \(f \colon P \to M\) where \(P\) is a projective module and \(\ker f\) is a superfluous submodule of \(P\). 

\begin{definition}\label{definition:Perfect}
    A ring \(R\) is called \emph{left (right) perfect} if every left (right) \(R\)-module has a projective cover. 
\end{definition}

A celebrated theorem of Bass~\cite[Theorem~P]{Bass_1960} characterizes left-perfect rings. Before stating Bass' theorem, recall that an \(R\)-module \(M\) is called \emph{flat} if the functor \(- \otimes_R M\) preserves injections. 

\begin{theorem}(Bass, 1960) \label{theorem:Bass}
    The following are equivalent:
    \begin{enumerate-(i)}
        \item \(R\) is left perfect;
        \item \(R\) satisfies the descending chain condition on principal right ideals;
        \item
        \label{Bass : 3}
        Every flat left \(R\)-module \(M\) is projective. 
    \end{enumerate-(i)}
\end{theorem}

It follows that the classes of flat and projective left (right) modules coincide in left (right) perfect rings. While the definition of flatness will not be used in this paper, it is crucial that projective modules are flat. (E.g., see~\cite[Theorem~4.3]{Lam_1998}.)

The following result is attributed to Barbara Osofsky by Sabbagh and Eklof in a postscript to their paper~\cite{Sabbagh_Eklof_1971}:

\begin{theorem}[Osofsky; Sabbagh--Eklof,~1971]\label{Theorem: Sabbagh, Eklof, Osofsky}
    Let $R$ be a ring. The class of projective $R$-modules is closed under ultraproducts if and only if $R$ is left perfect and right coherent.
\end{theorem}

Below we briefly outline the proof of the ``if'' direction, which is due to Sabbagh and Eklof~\cite{Sabbagh_Eklof_1971}.

\begin{proposition}
    \label{prop : reduced products flat}
Let \(\{M_i:i\in I\}\) be a family of \(R\)-modules and \(D\) be a filter on \(I\). If the direct product \(\prod_{i\in I} M_i\) is flat, then the reduced product \(\prod_{i\in I} M_i/D\) is flat.
\end{proposition}

\begin{proof}
    Let \(\pi_D\colon \prod_{i\in I} M_i \to \prod_{i\in I} M_i/D\) be the canonical projection and let \(K=\ker(\pi_D)\). We see that the short exact sequence \[0\to K\to\prod_{i\in I} M_i \to \prod_{i\in I} M_i /D \to 0\] is pure exact. Then we conclude that \(\prod_{i\in I} M_i/D\) is flat by \cite[Corollary~4.86]{Lam_1998}.
\end{proof}

\begin{theorem}[Sabbagh--Eklof,~1971]
    If \(R\) is left perfect and right coherent, then the ultraproduct of projective \(R\)-modules is projective.
\end{theorem}

\begin{proof}
    Suppose that \(R\) is left-perfect and right coherent, and let \(\{M_i:i\in I\}\) be any family of projective \(R\)-modules. In particular, each module \(M_i\) is flat. Since \(R\) is right coherent, the direct product \(\prod_{i\in I} M_i\) is flat. Then, by Proposition~\ref{prop : reduced products flat}, we conclude that any ultraproduct \(\prod_{i\in I} M_i/\mathcal{U}\) is flat. Since \(R\) is left-perfect, we conclude that \(\prod_{i\in I}M_i/\mathcal{U}\) is projective by Bass' characterization of left-perfectness. (See Theorem~\ref{theorem:Bass}\ref{Bass : 3}.)
\end{proof}

\begin{proof}[Proof of Theorem~\ref{strongresult!}] The statement follows from Lemma~\ref{lem:strong_compactness_result} and Theorem~\ref{Theorem: Sabbagh, Eklof, Osofsky}.
\end{proof}

Much to our regret, we could not find a proof of the other direction of Theorem~\ref{Theorem: Sabbagh, Eklof, Osofsky}. To the best of our knowledge, it was never published and only appeared in some personal communication
between B. Osofsky and the authors of \cite{Sabbagh_Eklof_1971} at the time when that paper was published. Nevertheless, we firmly believe this was a meaningful piece of mathematics and that such a gap in the literature needs to be filled.

Recall that whenever \(R\) is a p.i.d., an \(R\)-module \(M\) is free if and only if \(M\) is projective. Therefore, Theorem~\ref{result!} and Theorem~\ref{strongresult!} were already known in this particular case.
Also, whenever \(R\) is semisimple, all \(R\)-modules are projective. (See~\cite[Corollary~17.4]{Anderson_Fuller_1992}.)

However, there are examples of rings that are left-perfect, left-hereditary, and right-coherent (therefore, they satisfy the two conditions of Lemma~\ref{lem:strong_compactness_result}\ref{condition:(i)}--\ref{condition:(ii)}), but are not semisimple.

\begin{example}
    Let \(k\) be any field, and let \(T_2(k) =\{\begin{psmallmatrix}
        a&b\\
        0&c 
    \end{psmallmatrix}: a,b,c\in k\}\) the ring of upper triangular matrices over \(k\). It is well-known that \(T_2(k)\) is hereditary. (E.g., see~\cite[Example~2.36]{Lam_1998}.) Since \(T_2(k)\) is a finite dimensional \(k\)-algebra, \(T_2(k)\) is both left and right Artinian. (See~\cite[p.~39]{Lam_2001}.) 
    Right Artinian rings are necessarily right coherent.
    Moreover, since \(T_2(k)\) is right Artinian, \(T_2(k)\) satisfies the descending chain condition on (principal) right ideals. It follows that \(T_2(k)\) is left perfect by Theorem~\ref{theorem:Bass}.
    Last, straightforward computation shows that its Jacobson radical \(J \big( T_2(k)\big) = \{\begin{psmallmatrix}
        0&b\\
        0&0 
    \end{psmallmatrix}:b\in k\}\), thus \(T_2(k)\) is not semisimple by \cite[Theorem~4.14]{Lam_2001}.
\end{example}

\end{document}